\documentclass[12 pt]{amsart}
\usepackage{amssymb,amsfonts,amsthm,subfigure}
\usepackage{graphicx,caption}
\usepackage{bm}
\newtheorem{thm}{Theorem}[section]

\theoremstyle{definition}

\theoremstyle{remark}

\numberwithin{equation}{section}

\begin{document}
\title{A data-driven method for the steady state of randomly perturbed
  dynamics}
\author{Yao Li}
\address{Yao Li: Department of Mathematics and Statistics, 
University of Massachusetts Amherst, USA}
\email{yaoli@math.umass.edu}

\begin{abstract}
  We demonstrate a data-driven method to solve for the invariant
  probability density function of a randomly perturbed dynamical system. The
  key idea is to replace the boundary condition of numerical schemes
  by a least squares problem corresponding to a reference solution,
  which is generated by Monte Carlo simulation. With this method we
  can solve for the invariant probability density function in any local
  area with high accuracy, regardless of whether the attractor is
  covered by the numerical domain. 
\end{abstract}
\maketitle

\section{Introduction}
Many physical and biological systems are subject to random
perturbations. The time evolution of the probability density function of a
randomly perturbed dynamical system, i.e., a stochastic differential equation, is usually described by the Fokker-Planck
equation \cite{risken1996fokker}. In many studies, the invariant probability density function
of the randomly perturbed system is particularly important. It is well
known that under suitable conditions, the invariant probability
density function solves the steady state Fokker-Planck equation. On
the other hand, the positive solution to the steady state Fokker-Planck
equation must be an invariant probability density function of the
corresponding stochastic differential equation \cite{bogachev1999uniqueness, bogachev2009elliptic,
  huang2015steady}. 

When the unperturbed dynamical system has complex dynamics, an analytical
solution of the Fokker-Planck equation is usually not available. On the other
hand, numerically solving the steady state Fokker-Planck equation on an unbounded
domain is often challenging due to the lack of a well-posed boundary
condition. The usual practice is to let the numerical domain cover the global
attractor of the unperturbed dynamical system with sufficient
margin. The Freidlin-Wentzell theory \cite{freidlin1998random} 
guarantees that the invariant
probability density is close to zero when sufficiently far away from
the global attractor. Then it is usually safe to assume a zero
boundary condition. 

The resolution of the numerical solution imposes additional challenges. When the
strength of a random perturbation is $0 < \sigma \ll 1$, it is known that the
probability density function should concentrate on a
$O(\sigma)$-neighborhood of the attractor \cite{li2016systematic}. Hence the grid size of
the discretization cannot be larger than $\sigma$. Otherwise the
numerical scheme cannot ``see'' the concentration, and sometimes
serious numerical artifacts may occur. Therefore, when the noise
strength is small and the underlying dynamical system has complex
(possibly chaotic) dynamics, the grid size of the discretization has
to be sufficiently small. In addition, chaos only occurs in ordinary
differential equations in dimension $\geq 3$. This makes a numerical study
of interplays between chaos and random perturbations extremely
difficult. Take the Lorenz system for an
example. A very expensive numerical computation in \cite{allawala2016statistics} can
only solve the Fokker-Planck equation corresponding to the Lorenz
system on a $160\times 160 \times 160$ mesh, with a grid size $\approx
0.3$. 

The boundary condition is not a problem any more if one uses Monte Carlo
simulations to compute the invariant probability density function. A Monte Carlo simulation either runs
the stochastic differential equation for a long time, or runs many independent
trajectories of the stochastic differential equation for a finite amount of
time. The Monte Carlo simulation is an efficient way to obtain
statistics such as the expectation of a certain observable. However, the classical Monte Carlo simulation
has severe accuracy problem when solving for the invariant probability
density function. Unless one can generate a huge amount of samples, the probability density function generated by Monte Carlo
simulation is usually too ``noisy'' to be useful. 

In this paper we present a hybrid method that bypasses the
disadvantages of classical numerical PDE approach and the Monte Carlo
simulation. The key idea is to work on a domain without using any
boundary conditions. Then the discretization of a steady state Fokker-Planck equation becomes underdetermined,
which essentially gives a linear
constraint. Instead of the boundary condition, we generate an
approximated invariant probability density function by using the Monte Carlo
simulation. This approximated density function does not have to be
very accurate, because it only serves as a reference of the next step. Finally, we solve
a least squares optimization problem under the linear constraint given
by the discretization. The resultant solution satisfies the
discretization of the steady state Fokker-Planck equation (without a
boundary condition), and has minimum $L^{2}$ distance to the
approximated density function generated by the Monte Carlo simulation. The least
squares problem is further converted to a linear system that can be
solved either exactly or using iterative methods. This method can help
us to compute a high resolution solution in a local area without
worrying about boundary
conditions. 

We demonstrate several numerical examples in this paper. The 1D
double-well potential is used to test the accuracy and performance of
the algorithm. The overall accuracy is satisfactory considering the performance. Then we
demonstrate the strength of this method with 2D and 3D examples. In the 2D
example, we show that a transition from relaxation oscillations to a
smaller limit cycle is destroyed by small noise. A local solution with high
resolution is presented to demonstrate some interesting local
structures in the invariant probability density function. In 3D
examples, we compute invariant probability density functions of small
random perturbations of two chaotic oscillators, the
Lorenz oscillator and the R\"ossler oscillator. With low computation
cost, we are able to find numerical
solutions with much higher resolution (grid size
$= 0.05$) than in previous studies. 

We remark that the purpose of this paper is only to introduce a
general framework. The detailed implementation can be further
improved in many ways. For example, a divide-and-conquer strategy can
significantly improve the performance of this hybrid algorithm. The naive Monte Carlo simulation can
be replaced by various importance sampling techniques
\cite{robert2004monte, tokdar2010importance}. If
the noise is large enough to smear fine structures, the high
dimensional Monte Carlo
sampler proposed in \cite{chen2017beating,
  chen2018efficient} can be adopted to our framework. The finite difference discretization can be replaced by other advanced solvers like the finite element method or
other methods for high dimensional problems \cite{naprstek2016multi,
  sun2015numerical, wang2016sparse}.  We will
write several subsequent papers to address these issues.

\section{Probability and numerics preliminary}
\subsection{Problem setting}
Consider an autonomous ordinary differential equation 
\begin{equation}
  \label{ODE}
x' = f(x) ,\quad x \in \mathbb{R}^{n} \,.
\end{equation}
We are particularly interested in situations when equation \eqref{ODE} generates
non-trivial dynamics. For example, equation \eqref{ODE} may admit a
strange attractor or have
separation of time scales that leads to interesting dynamics like the folding
singularity, mixed mode oscillations
etc. \cite{guckenheimer2005canards, desroches2012mixed}.

Now we consider the following dynamical system with random
perturbations, i.e., a stochastic differential equation (SDE), 
\begin{equation}
\label{SDE}
  \mathrm{d}X_{t} = f(X_{t}) \mathrm{d}t + \sigma(X_{t}) \mathrm{d}W_{t} \,,
\end{equation}
where $X_{t} \in \mathbb{R}^{n}$, $f: \mathbb{R}^{n} \rightarrow
\mathbb{R}^{n}$ is a vector field, $\sigma: \mathbb{R}^{n} \rightarrow
\mathbb{R}^{n\times n}$ is a matrix-valued function, and $\mathrm{d}W_{t}$ is 
the $n$-dimensional white noise. Throughout this paper,  we assume that
equation \eqref{SDE} admits a unique diffusion process solution
$X_{t}$. (See {\bf (H)} below for the full assumption.) Note that the
existence and the uniqueness of $X_{t}$
follow from mild assumptions on $f$ and $\sigma$, e.g., Lipschitz continuity of $f$ and
$\sigma$ \cite{oksendal2003stochastic, karatzas2012brownian}.

Since $X_{t}$ is a diffusion process, we denote the transition kernel of $X_{t}$ by $P^{t}(x, A) =
\mathbb{P}[X_{t} \in A \, | \, X_{0} = x]$. A probability measure
$\pi$ is said to be {\it invariant} if $\pi P^{t} = \pi$, where the
left operator is defined as
$$
 \pi P^{t}(A) = \int_{\mathbb{R}^{n}} P^{t}(x, A) \pi( \mathrm{d}x) \,. 
$$

It is well known that the time evolution of the probability density
function of $X_{t}$, denoted by $u_{t}$, is
described by the Fokker-Planck equation
\begin{equation}
\label{FPE}
  u_{t} = \mathcal{L}u = -\sum_{i = 1}^{n} (f_{i}u)_{x_{i}} +
  \frac{1}{2}\sum_{i,j = 1}^{n} (D_{i,j}u)_{x_{i}x_{j}} \,, \quad u(0,
  x) = u_{0}(x) \,,
\end{equation}
where $D = \sigma^{T}\sigma$, $u_{0}(x)$ is the probability density
function of $X_{0}$. A
probability density function $u_{*}(x)$ is said to be an {\it
invariant probability density function} if $\mathcal{L}u_{*} = 0$. It
is easy to see that an invariant probability density function defines an invariant probability measure $\pi$ of
$X_{t}$, and $u_{*}(x)$ is the probability density function of $\pi$.

The existence of an invariant probability measure is guaranteed if
$X_{t}$ is defined on a compact manifold without boundary \cite{zeeman1988stability}. When
$X_{t}$ is defined on unbounded domain, such existence needs some
``dissipation'' conditions \cite{bogachev2001generalization,
  huang2015steady, khasminskii2011stochastic}. The uniqueness of $\pi$
usually follows if $D$ is non-degenerate (everywhere positive definite). The convergence to the invariant
probability measure is another tricky issue. To make $P^{t}(x, \cdot)
\rightarrow \pi$ as $t \rightarrow \infty$, one needs stronger
``dissipation'' conditions and some minorization-type conditions
\cite{meyn2012markov, hairer2006ergodicity, khasminskii2011stochastic}. Since
the theme of this paper is to introduce a numerical algorithm, we have the following
assumption on $X_{t}$, $P^{t}(x, \cdot)$, and $\pi$ throughout the paper. 

\medskip
{\bf (H)} {\it Equation \eqref{SDE} admits a unique diffusion process
  $X_{t}$. The diffusion process $X_{t}$ has a unique invariant
probability measure $\pi$ that is absolutely continuous with respect
to the Lebesgue measure. The probability density function of $\pi$ uniquely
solves the stationary Fokker-Planck equation. In addition, $P^{t}(x,
\cdot) \rightarrow \pi$ as $t\rightarrow \infty$ for every $x \in
\mathbb{R}^{n}$. }

\subsection{Numerical PDE approach for computing invariant measure.}

There are two different approaches for computing $u_{*}$. One can
either solve the Fokker-Planck equation \eqref{FPE} up to a
sufficiently large $t$, or
solve the stationary Fokker-Planck equation directly. The biggest
problem of the numerical PDE approach is the boundary condition. For
the sake of simplicity, we illustrate the problem by using the finite
difference scheme. The case of the finite element scheme is analogous. 

Without loss of generality, we solve the Fokker-Planck equation
numerically on a 2D domain $[-L, L]^{2}$. Let the spatial and time step sizes be
$r = 2L/N$ and $h$, respectively. Let $\mathbf{u}^{m} = \{u^{m}_{i,j}\}_{i, j = 0}^{N}$ be the
discretized solution at time step $m$. Entry $u^{m}_{i,j}$ is a numerical
approximation of $u(mh, r i - L, rj -L)$. We need boundary conditions to
update the solution to $\mathbf{u}^{m+1}$. The usual approach is to
let the domain cover the global attractor of the ODE \eqref{ODE} with sufficient margin. Then we assume zero boundary conditions
$u^{m}_{0, j} = u^{m}_{N,j} = u^{m}_{i, 0} = u^{m}_{i, N} = 0$ and
compute $\mathbf{u}^{m+1}$ by using either implicit Euler scheme or Crank-Nickson
scheme. Since the probability of $X_{t} \notin [-L, L]^{2}$ is nonzero,
$\mathbf{u}^{m}$ needs to be renormalized after each update such that
$$
  \sum_{i, j = 0}^{N} u^{m}_{i, j} = \frac{1}{r^{2}} \,.
$$
When $\|\mathbf{u}^{m} - \mathbf{u}^{m-1}\|$ is smaller than the error
tolerance $\epsilon$ for some $m = M$, the update is stopped. Now $\mathbf{u}^{M}$
numerically solves the steady state Fokker-Planck equation. In order to make the
numerical solution reliable, the domain has to be sufficiently large
such that $\mathbb{P}[X_{t} \notin  [-L, L]^{2}] \ll 1$ for $t \in [0,
Mh]$ and $1 - \pi([-L, L]^{2}) \ll 1$.

Another approach is to solve the steady state Fokker-Planck equation
directly. Assume the same 2D domain as before. In order to discretize
$\mathcal{L}u_{*}$, the boundary value of $u_{*}$
on $\partial [-L, L]^{2}$ is necessary. The usual practice is to make $L$ large
enough so that $[-L, L]^{2}$ covers the global attractor of equation
\eqref{ODE} with sufficient margin. Then we can assume a zero boundary
condition because of the Freidlin-Wentzell theory \cite{freidlin1998random}. This will
generate a linear system
$$
  A \mathbf{u} = \mathbf{0} \,,
$$
where $A$ is an $n \times n$ nonsingular matrix. To avoid the trivial
solution, one also needs the constraint 
$$
  \mathbf{1}^{T} \mathbf{u} = r^{-2} \,.
$$
This gives an overdetermined linear system
\begin{eqnarray}
\label{lsq1}
A \mathbf{u} & = & \mathbf{0}\\\nonumber
\mathbf{1}^{T} \mathbf{u} &=& r^{-2} \,.
\end{eqnarray}
Let
$$
  \hat{A} = 
\begin{bmatrix}
A\\
\mathbf{1}^{T}
\end{bmatrix}
\quad , \quad \mathbf{b} = 
\begin{bmatrix}
\mathbf{0}\\ r^{-2}
\end{bmatrix} \,.
$$
We can find the least squares solution $\hat{\mathbf{u}}$ that solves the
optimization problem
$$
\min  \| \hat{A} \hat{\mathbf{u}} - \mathbf{b} \|_{2} \,.
$$
The least squares solution $\hat{\mathbf{u}}$ numerically solves the
steady state Fokker-Planck equation.

\subsection{Probabilistic approach for computing invariant measure.}

The numerical PDE approach works reasonably well for 1D and 2D
problems. However, in higher dimension this approach becomes not
practical. In particular, the domain has to be sufficiently large to
cover the global attractor of equation \eqref{ODE} with enough
margin. This imposes great difficulty to many practical problems. For
example, if equation \eqref{ODE} is a Lorenz system, then we need a
grid in a $50 \times 50 \times 50$ box to cover the attractor. (See section 4.3
for more discussion about the Lorenz system.)

An alternative approach is to use Monte Carlo simulation. One can
collect samples of $X_{t}$ over a long trajectory in any dimension,
although the accuracy of the Monte Carlo simulation suffers greatly from
the curse-of-dimensionality. Let $h \ll
1$ be the step size. Let $X_{n} := X_{nh}$ be the numerical time-$h$
sample chain produced by certain numerical method (Euler, Milstein, 
Runge-Kutta .etc) \cite{kloeden2013numerical}. Under certain conditions, $X_{n}$ admits an invariant probability
measure $\pi_{h}$ that converges to $\pi$ as $h \rightarrow 0$
\cite{mattingly2010convergence, mattingly2002ergodicity}. In
addition, $X_{n}$ is a Markov chain. Let $\xi : \mathbb{R}^{n}
\rightarrow \mathbb{R}$ be an observable on $\mathbb{R}^{n}$ and
$\mathbf{N}$ be the number of samples. By the
law of large numbers of Markov chains \cite{meyn2012markov}, we have
$$
  \frac{1}{\mathbf{N}}\sum_{n = 1}^{\mathbf{N}} \xi(X_{n}) \rightarrow \pi(\xi) \quad
  a.s. \,.
$$

Therefore, we can use Monte Carlo simulation to compute the
probability density function of $\pi$. For the sake of simplicity we
consider grid points $\mathbf{u} =  \{u_{i,j}\}_{i, j = 0}^{N}$ in
a 2D domain $[-L, L]^{2}$, such that $u_{i,j}$ is the numerical approximation
of $u_{*}(ir - L, j r - L)$. Let $O_{i,j} = [ir - L - r/2, ir - L +
r/2] \times [jr - L - r/2, jr - L + r/2]$. Then the Monte Carlo
simulation gives
$$
  u_{i,j} = \frac{1}{\mathbf{N} r^{2}}\sum_{n = 1}^{N}
  \mathbf{1}_{O_{i,j}}(X_{n}) 
$$
for some sufficiently large $\mathbf{N}$. In practice, we construct $(N+1)^{2}$
boxes $O_{i,j}$ and simulate $X_{n}$ over a long time period. After
the simulation, $u_{i,j}$ is obtained by counting sample points of
$X_{n}$ falling into $O_{i,j}$.

It is easy to see that the Monte Carlo simulation approach has a significant
disadvantage on the accuracy because it is difficult to collect enough
sample points in each $O_{i,j}$. Without loss of generality, we assume 
$$
  \sum_{n = 1}^{\mathbf{N}}  \mathbf{1}_{O_{i,j}}(X_{n})  = O(r^{2}\mathbf{N}) \,.
$$
If we treat $\mathbf{1}_{O_{i,j}}(X_{n})$ as i.i.d Bernoulli random
variables, some calculation shows that the standard deviation of
$u_{i,j}$ is $O(r^{-1} \mathbf{N}^{-1/2})$. Hence $\mathbf{N}$ has to be
very large to control the standard deviation. For example,
$\mathbf{N}$ needs to be $O(r^{-5})$ to reduce the standard deviation to
$O(r^{2})$. The accuracy problem
will be much worse in higher dimensions. In practice, the solution
obtained from the Monte Carlo simulation usually looks very ``noisy''. 

Despite of its accuracy problem, the Monte Carlo simulation has more
flexibility because of the following reasons. (1) In the Monte Carlo
simulation, the domain $[-L, L]^{2}$ does not have to cover the global
attractor of \eqref{ODE}. (2) The curse of dimensionality is
slightly alleviated if equation \eqref{ODE} has a lower dimensional global
attractor. Because the invariant probability measure $\pi$
concentrates on the vicinity of the global attractor of equation \eqref{ODE}. (3)
Parallel computing is much easier for Monte Carlo simulations. (4)
Some high dimensional sampling technique can be applied to 
improve the Monte Carlo simulation for a large class of dynamical systems.

\section{A hybrid data-driven method}

We propose the following hybrid method that combines the high accuracy
of the numerical PDE approach and the flexibility of the Monte Carlo
simulation. Consider a 2D domain 
$[a_{0}, b_{0}] \times [a_{1}, b_{1}]$ that does not have to cover any
attractor of equation \eqref{ODE}. Let $\mathbf{u} = \{u_{i,j}\}_{i =
  1, j = 1}^{i = N, j = M}$ be the numerical solution. Without loss of
generality assume $r = (b_{0} - a_{0})/N  = (b_{1} - a_{1})/M$. Then
$u_{i,j}$ approximates $u_{*}$ at the grid point $(i r + a_{0}, jr +
a_{1})$. By discretizing the steady state Fokker-Planck equation
$\mathcal{L}u_{*} = 0$ {\it without} any boundary condition, we have a
linear system with normalizing condition
\begin{equation*}
\left \{
\begin{array}{ccc}
B \mathbf{u}& =  & 0 \\
\mathbf{1}^{T} \mathbf{u}& =& r^{-2} \,.
\end{array}
\right .
\end{equation*}
Let
$$
  \hat{B} = 
\begin{bmatrix}
B\\
\mathbf{1}^{T}
\end{bmatrix}
\quad , \quad \mathbf{b} = 
\begin{bmatrix}
\mathbf{0}\\ r^{-2}
\end{bmatrix} \,.
$$
Obviously $\hat{B}$ is not a full-rank matrix. Hence we obtain a
linear constraint $\hat{B} \mathbf{u} = \mathbf{b}$.

Then we run the Monte Carlo simulation to get another approximate
solution $\mathbf{v} = \{v_{i,j}\}_{i =
  1, j = 1}^{i = N, j = M}$. Let $O_{i,j} = [a_{0} + ir - r/2, a_{0} + ir + r/2] \times
[a_{1} + jr - r/2, a_{1} + jr + r/2] $. Let $\mathbf{N}$ be a 
large number, the Monte Carlo simulation gives
$$
  v_{i,j} = \frac{1}{\mathbf{N} r^{2}}\sum_{n = 1}^{N}
  \mathbf{1}_{O_{i,j}}(X_{n})  \,.
$$
The approximate solution $\mathbf{v}$ in this step does not have to be
very accurate. We will use it in the next step to obtain a much more
accurate numerical solution $\mathbf{u}$. 

The key step of this hybrid solution is to solve the following optimization
problem. We combine the linear constraint with the ``noisy'' data
$\mathbf{v}$ obtained from the Monte Carlo simulation. The idea is
that $\mathbf{u}$ should both satisfy the linear constraint from the
discretization and be as close to $\mathbf{v}$ as
possible. This leads to the optimization problem
\begin{eqnarray}
\label{lsq}
 &  \mbox{min}& \| \mathbf{u} - \mathbf{v} \|_{2} \\\nonumber
&\mbox{subject to }& \hat{B} \mathbf{u} = \mathbf{b} \,.
\end{eqnarray}
Let $\mathbf{x} = \mathbf{u} - \mathbf{v}$, this reduces to the problem
\begin{eqnarray}
\label{lsqnew}
 &  \mbox{min}& \| x \|_{2} \\\nonumber
&\mbox{subject to }& \hat{B} \mathbf{x} = \mathbf{d} \,,
\end{eqnarray}
where 
$$
  \mathbf{d} = \mathbf{b} - \hat{B} \mathbf{v} \,.
$$
The following theorem is a straightforward textbook result. (See for
example \cite{boyd2004convex}.) We include the proof of the sake of
completeness of the paper. 

\begin{thm}
If $\hat{B}$ has linearly independent rows, then
\begin{equation}
\label{normal}
  \mathbf{\hat{x}} = \hat{B}^{T}(\hat{B} \hat{B}^{T})^{-1} \mathbf{d} 
\end{equation}
is the unique solution of \eqref{lsqnew}.
\end{thm}
\begin{proof}
It is easy to see that $\mathbf{\hat{x}}$ solves the linear constraint
$\hat{B} \mathbf{\hat{x}} = \mathbf{d}$. 

For any vector $\mathbf{x} \neq \mathbf{\hat{x}}$ satisfying $\hat{B}
\mathbf{\hat{x}} = \mathbf{d}$, we have
\begin{eqnarray*}
\|\mathbf{x}\|^{2} & =  & \| \mathbf{\hat{x}} + \mathbf{x} -
                          \mathbf{\hat{x}} \|^{2} \\
&=& \| \mathbf{\hat{x}}\|^{2} + 2 \mathbf{\hat{x}}^{T}( \mathbf{x} -
    \mathbf{\hat{x}}) + \| \mathbf{x} - \mathbf{\hat{x}} \|^{2} \,.
\end{eqnarray*}
Since $\hat{B} \mathbf{x} = \mathbf{d}$, we have
$$
   \mathbf{\hat{x}}^{T}( \mathbf{x} -
    \mathbf{\hat{x}}) = \mathbf{d}^{T}(\hat{B}^{T}\hat{B})^{-1}
    \hat{B}( \mathbf{x} - \mathbf{\hat{x}}) = 0 \,.
$$
Therefore, we have
$$
  \|\mathbf{x}\|^{2} = \| \mathbf{\hat{x}}\|^{2} + \| \mathbf{x} -
  \mathbf{\hat{x}} \|^{2} \geq \| \mathbf{\hat{x}} \|^{2} \,.
$$
This completes the proof.
\end{proof}

An efficient way to solve equation \eqref{normal} is to use the QR
factorization $\hat{B}^{T} = QR$. After a QR factorization, we have 
$$
  \mathbf{\hat{x}} = Q (R^{-1})^{T} \mathbf{d} \,.
$$
Then $\mathbf{u} = \mathbf{\hat{x}} + \mathbf{v}$ is the desired numerical
solution.

Heuristically, it is easy to see that the optimization problem
\eqref{lsq} can significantly reduce the error of the data
$\mathbf{v}$ generated by the Monte Carlo
simulation. Assume the Monte Carlo sampler does not have significant
bias. Let $\mathbf{u}_{*} = \{u_{*}^{i,j}\}_{i = 1, j = 1}^{i = N, j = M}$ be the vector corresponding to the exact
solution, i.e., $u_{*}^{i,j} = u_{*}(ir + a_{0}, jr + b_{0})$. Then the error term $\mathbf{w} := \mathbf{v} - \mathbf{u}_{*}$ can be
approximated by a random vector such that each entry has zero
expectation and finite variance. Let $\hat{\mathbf{u}}$ be the
solution to the optimization
problem \eqref{lsq}. Then the new error term $\mathbf{\hat{w}} := \hat{\mathbf{u}} - \mathbf{u}_{*}$ is
approximated by $\mathcal{P}(\mathbf{v} - \mathbf{u}_{*})$, where
$\mathcal{P}$ is the projection operator corresponding to the hyperplane
given by $\hat{B} \mathbf{u} = \mathbf{b}$. With high probability, the
projection of a random vector $\mathbf{w}$ to a much lower dimensional
hyperplane has
much smaller norm. In other words we have $\|\hat{\mathbf{w}}\| \ll \|
\mathbf{w}\|$ with high probability. The full proof
is much longer than the above heuristic description. Since the aim of
this paper is to introduce the algorithm, we prefer to put the rigorous proof
about this algorithm into our subsequent paper.

Finally, we remark that one significant advantage of this approach is that we can obtain a high
resolution solution in any local area. There is no restriction on the
domain as long as the Monte Carlo simulation can produce enough sample
points. If a global solution is necessary, we can
divide the space into many subdomains $I_{1}, \cdots, I_{K}$, solve them separately, and combine
them together according to the Monte Carlo simulation. (The
probability of a subdomain $\pi(I_{k})$ can be obtained from the Monte
Carlo simulation, which is the weight of $I_{k}$ when generating the
global solution.) This divide-and-conquer strategy allows us to solve
large scale problems ($\sim 10^{9}$ grid points) on a laptop. We will address it
in full details in our subsequent paper.

\section{Numerical examples}
We will illustrate our hybrid approach with three examples: the double
well potential gradient flow, the Van der Pol oscillator, and 3D chaotic
oscillators.
\subsection{Double-well potential and error analysis}
The first example is the gradient flow with respect to a double-well
potential. Consider the potential function 
$$
  U = \frac{1}{2} x^{4} - x^{2} 
$$
and the stochastic differential equation
$$
  \mathrm{d}X_{t} = - U'(X_{t}) \mathrm{d}t + \sigma \mathrm{d}W_{t}
$$
for $\sigma = 0.6$. The probability density function of this system
is 
$$
  u_{*}(x) = \frac{1}{K}e^{-2U(x)/\sigma^{2}} \,,
$$
where $K$ is a normalizer. 

In this example, we demonstrate our method by solving for $u_{*}(x)$ on the interval of
$[0, 2]$. Note that this system has two equilibria at $\pm
1$. Equilibrium  $x = -1$ is not covered by the numerical domain $[0, 2]$. In
particular, we have $u_{*}(0) = 0.1062$ so the zero boundary
condition does not apply to this problem. We also show the accuracy
and computational time of the hybrid method in this example. 

Let $X_{n}$ be the discrete-time Markov chain given by the
Euler–Maruyama method with a fixed time step size $\mathrm{d}t = 0.001$. The
Monte Carlo simulation is done by simulating $X_{n}$ up to time $T = \mathbf{N}h$. The
numerical solution from the optimization problem, the exact probability density function $u_{*}(x)$, and the
approximate density function from Monte Carlo simulation are compared in
Figure \ref{1d}. We can see that the Monte Carlo
simulation itself only produces a solution with low accuracy, unless
one can collect a huge amount of samples. In fact, when $h =
0.005$, $T$ needs to be at least $10^{5}$ to remove the sawtooth in
the solution (the orange plot in Figure \ref{1d}). The rough solution
generated by the Monte Carlo simulation is smoothed and corrected by
the linear constraint (the red plot in Figure \ref{1d}).

\begin{figure}[htbp]
\centerline{\includegraphics[width = \linewidth]{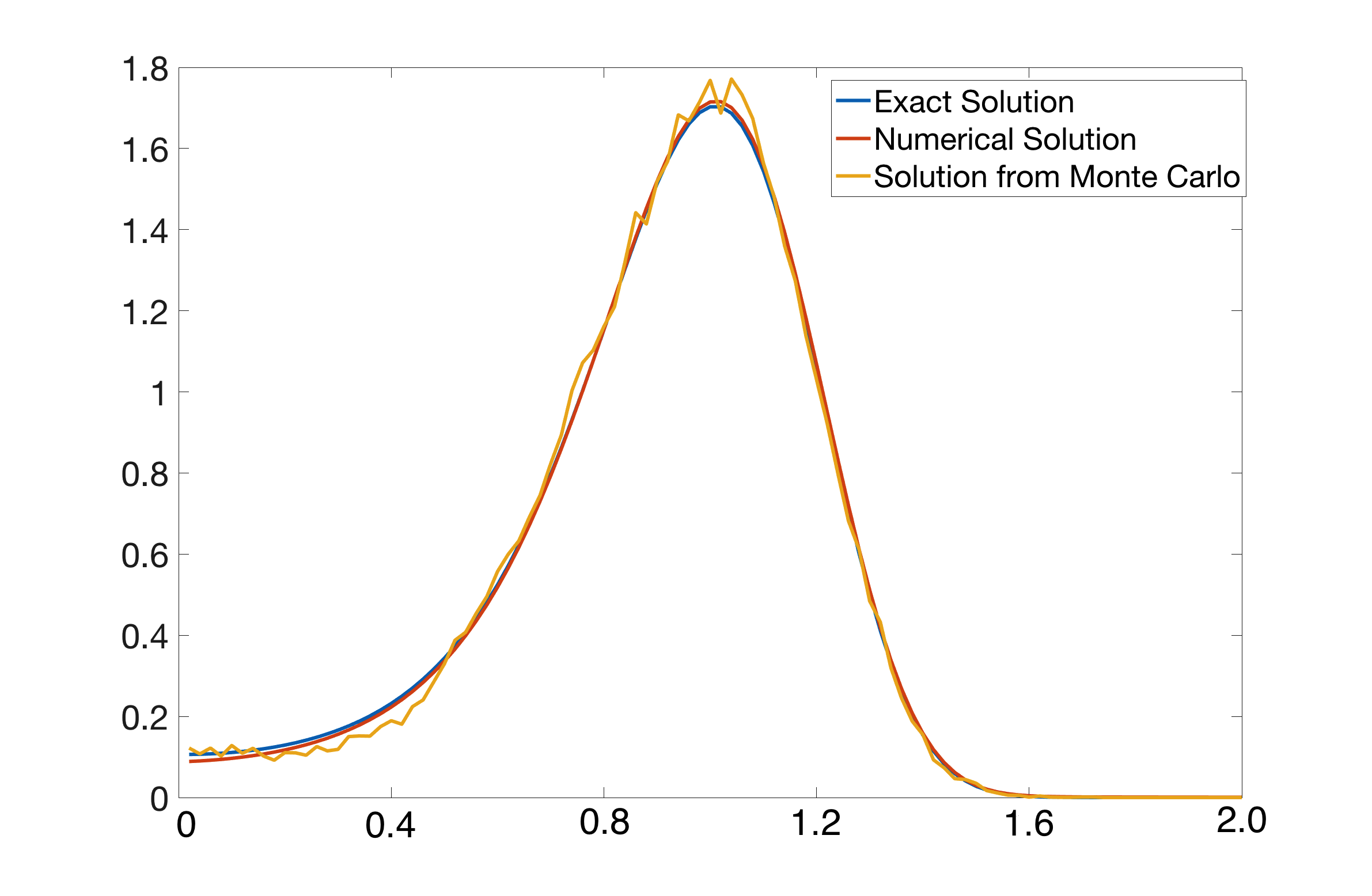}}
\caption{A comparison of exact solution, solution from Monte Carlo,
  and solution from the hybrid method.}
\label{1d}
\end{figure}

The following table
compares $L^{2}$ errors with varying time span $T$ and grid size $h$. Each
entry is the average $L^{2}$ error of $5$ trials. We can see that
despite some randomness caused by the Monte Carlo simulation, the
error drops with smaller grid size and larger sample size for the Monte Carlo
simulation. Note that the invariant probability measure of $X_{n}$ is
only an $O(\mathrm{d}t) = 1.0 \times 10^{-3}$ 
approximation of that of $X_{t}$. Hence it does not make
sense to test the accuracy with more samples or smaller grid
size unless one makes the time step even smaller. In this small-scale
1D problem, the classical numerical PDE solver is more accurate,
mainly because the invariant probability measure of $X_{n}$ is only a
first order approximation of that of $X_{t}$. But overall the accuracy is
satisfactory given the performance of the algorithm. 

\begin{center}
\begin{table}
\begin{tabular}[h]{|c|c|c|c|c|}
\hline
$T \setminus h$ & 0.04&0.02&0.01&0.005\\
\hline
$500$& $6.233\times 10^{-3}$ & $5.971 \times 10^{-3}$ & $2.260 \times
                                                    10^{-3}$ & $2.120
                                                               \times
                                                               10^{-3}$\\
\hline
$1000$ & $4.289\times 10^{-3} $&$5.679 \times 10^{-3}$&$2.137\times
                                                      10^{-3}$&$2.893\times
                                                                10^{-3}$\\
\hline
$2000$& $2.240\times 10^{-3}$&$2.562\times 10^{-3}$&$2.483\times
                                                   10^{-3}$&$1.476\times
                                                             10^{-3}$\\
\hline
$4000$&$1.914\times 10^{-3}$&$2.159\times 10^{-3}$&$1.339\times
                                                  10^{-3}$&$0.656\times
                                                            10^{-3}$\\
\hline
\end{tabular}
\caption{Accuracy of the hybrid method with respect to different sample
sizes and grid sizes.}
\label{table1}
\end{table}
\end{center}

It remains to comment on the computation time. We choose smaller grid
sizes $h = 2 \times
10^{-4}$, $1 \times 10^{-4}$, and $5 \times 10^{-5}$ to
highlight the difference between different approaches. The time span
of Monte Carlo simulation is chosen to be $T = 4000$. In order to
apply the numerical PDE approach directly, one needs to enlarge the
domain to $[-2, 2]$ to apply the zero boundary condition. The
direct Monte Carlo simulation is too slow to be interesting if one wants to achieve the
same accuracy. Hence we only compare the numerical PDE approach and the 
hybrid method in the following table. The computation time for
the hybrid method is further broken down in to the Monte Carlo
simulation phase (Phase 1) and the optimization phase (Phase 2). The
numerical PDE approach uses the MATLAB solver {\it mldivide} (the backslash solver). The Monte
Carlo simulation is written in C++. The optimization problem equation
\eqref{lsq} is solved by the MATLAB solver {\it lsqminnorm}. 

\begin{center}
\begin{table}
\begin{tabular}[h]{|c|c|c|c|c|}
\hline
grid size $h$ & Numerical PDE&Hybrid (Total)&Hybrid (phase 1)&Hybrid (phase 2)\\
\hline
$2\times 10^{-4}$ & $0.05943$ sec &$0.2775$ sec &$0.2771$ sec &
                                                                $0.004236$
                                                                sec \\
\hline
$1 \times 10^{-4}$ &$0.2565$ sec&$0.2668$ sec &$0.2663$ sec &
                                                              $0.004887$
                                                              sec \\
\hline
$5\times 10^{-5}$ & $1.111$ sec& $0.2618$ sec & $0.2609$ sec &
                                                               $0.009218$
                                                               sec \\
\hline
\end{tabular}
\caption{Computational time for the numerical PDE approach and the
  hybrid method. }
\label{table2}
\end{table}
\end{center}

From Table \ref{table2}, we can find that solving the optimization
problem \eqref{lsq} is actually much faster than solving the least
squares problem \eqref{lsq1}. This is because the numerical PDE approach
has to cover both equilibria with considerable margin in order to apply the zero
boundary condition. In addition, solving the overdetermined least
squares problem in equation \eqref{lsq1} takes more time for the MATLAB
solver we use. In this problem, the Monte Carlo simulation is the bottleneck of the hybrid
method. In higher dimensional problems, such as the Lorenz oscillator in Section 4.3, solving
the optimization problem \eqref{lsq} usually takes much more time than
the Monte Carlo simulation. And the classical numerical PDE approach becomes
not practical any more, because the domain still has to cover the entire
attractor in order to apply the zero boundary condition.

\subsection{Van der Pol oscillator and canard}

The second example is a Van der Pol oscillator, which has been
intensively used in both physics and mathematical biology
\cite{guckenheimer1980dynamics, kanamaru2007van}. In this
subsection, we use our method to demonstrate an interesting phenomenon
related to the canard solution. 

Consider an oscillator 
\begin{eqnarray}
\label{canard}
 \dot{x}& =   &\frac{1}{\epsilon}(y - \frac{1}{3}x^{3} + x) \\\nonumber
\dot{y} &=& a - x \,,
\end{eqnarray}
where $\epsilon = 0.1$ is the time scale separation parameter, and $a$
is a control parameter. This system is a prototypical example for the
canard explosion. A ``canard'' is a solution that the system can pass
a bifurcation point of the critical manifold and follow the repelling
part of the slow manifold for some amount of time \cite{benoit1990canards}. Usually a canard
solution only exists for a very small range of parameters. 

\begin{figure}[htbp]
\centerline{\includegraphics[width = \linewidth]{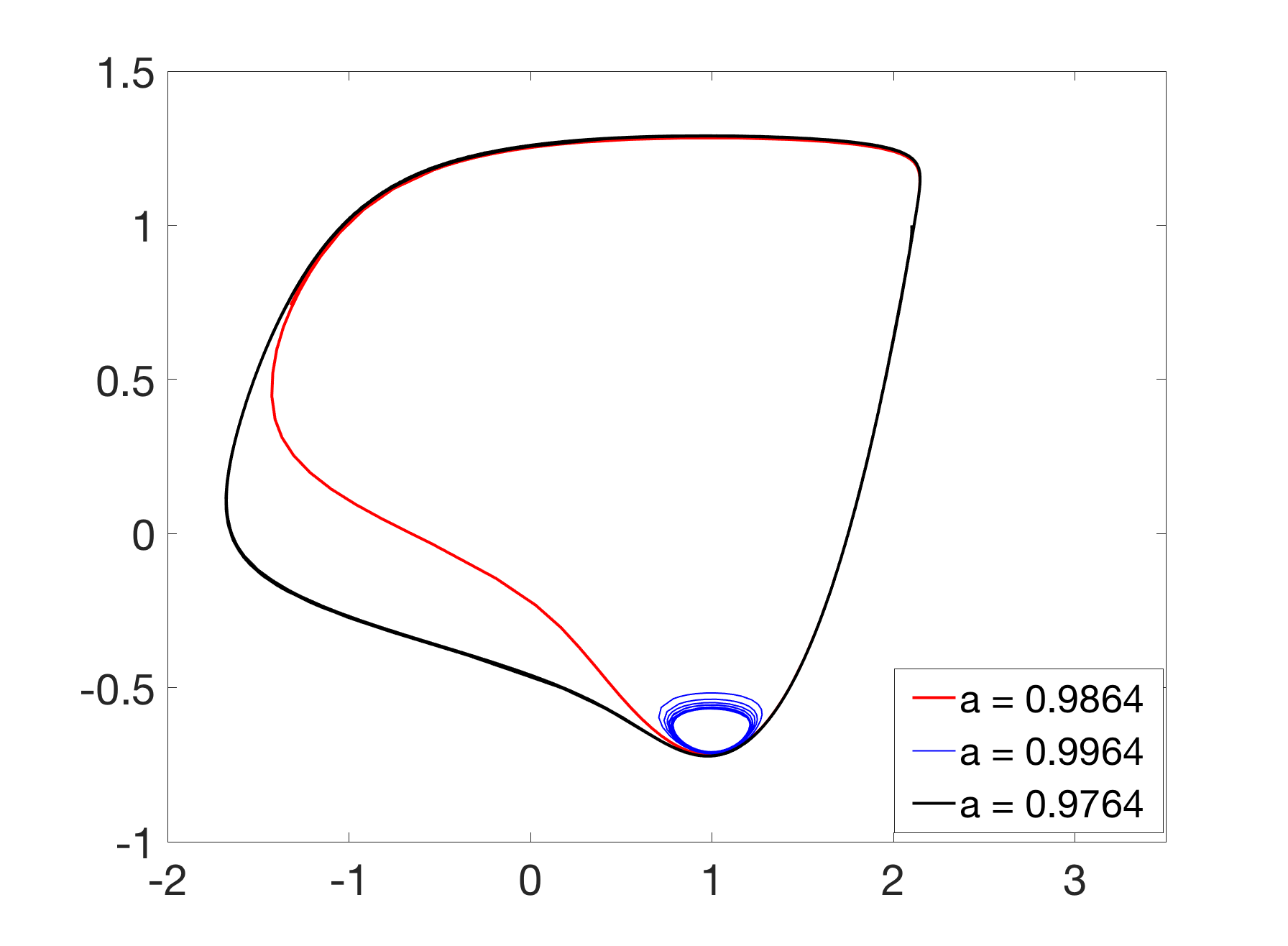}}
\caption{Bifurcation of the Van der Pol oscillator. When $a$ is small,
  the system demonstrates relaxation oscillations (Black). The red solution
  is a canard solution, at which the solution follows the repelling
  part of the slow manifold for a short period of time. With
  increasing $a$, a transition through the canard occurs and the
  solution follows a smaller limit cycle (Blue).}
\label{canarddet} 
\end{figure}

We consider the random perturbation of system \eqref{canard}
\begin{eqnarray}
\label{canardsde}
 \mathrm{d}X_{t}& =   &\frac{1}{\epsilon}(Y_{t} - \frac{1}{3}X_{t}^{3}
                        + X_{t})  \mathrm{d}t+ \sigma \mathrm{d}W^{1}_{t}\\ \nonumber
\mathrm{d}Y_{t} &=& (a - X_{t}) \mathrm{d}t + \sigma \mathrm{d}W^{2}_{t}  \,,
\end{eqnarray}
where  $W^{1}_{t}$ and $W^{2}_{t}$ are two independent Wiener processes. The
parameter $a$ is chosen to be $0.9964$, at which the deterministic
system has
already passed the canard bifurcation and is attracted to a smaller limit
cycle (blue curve in Figure \ref{canarddet}). We use our hybrid method to
compute the density function $u_{*}(x)$ of the invariant probability measure of
system \eqref{canardsde}. Our numerical result shows that this
transition through a canard solution is essentially destroyed by
a small random perturbation. Although the deterministic system admits a
smaller limit cycle, the steady state probability density function $u_{*}(x)$
still concentrates on the large limit cycle corresponding to the relaxation
oscillations (the black curve in Figure \ref{canarddet}). When the strength of noise increases, the support of $u_{*}(x)$ not
only becomes ``wider'', but also has significant
deformation. When the noise is large ($\sigma = 1.0$), some
probability density moves to the slow manifold that does not belong to
any limit cycle of the deterministic system and forms two
``tails''. With the hybrid method
introduced in this paper, we can get a high resolution local solution
about the lower left ``tail'' with much higher precision (Panel 6 of
Figure \ref{canardlocal}). 

\begin{figure}[htbp]
\centerline{\includegraphics[width = \linewidth]{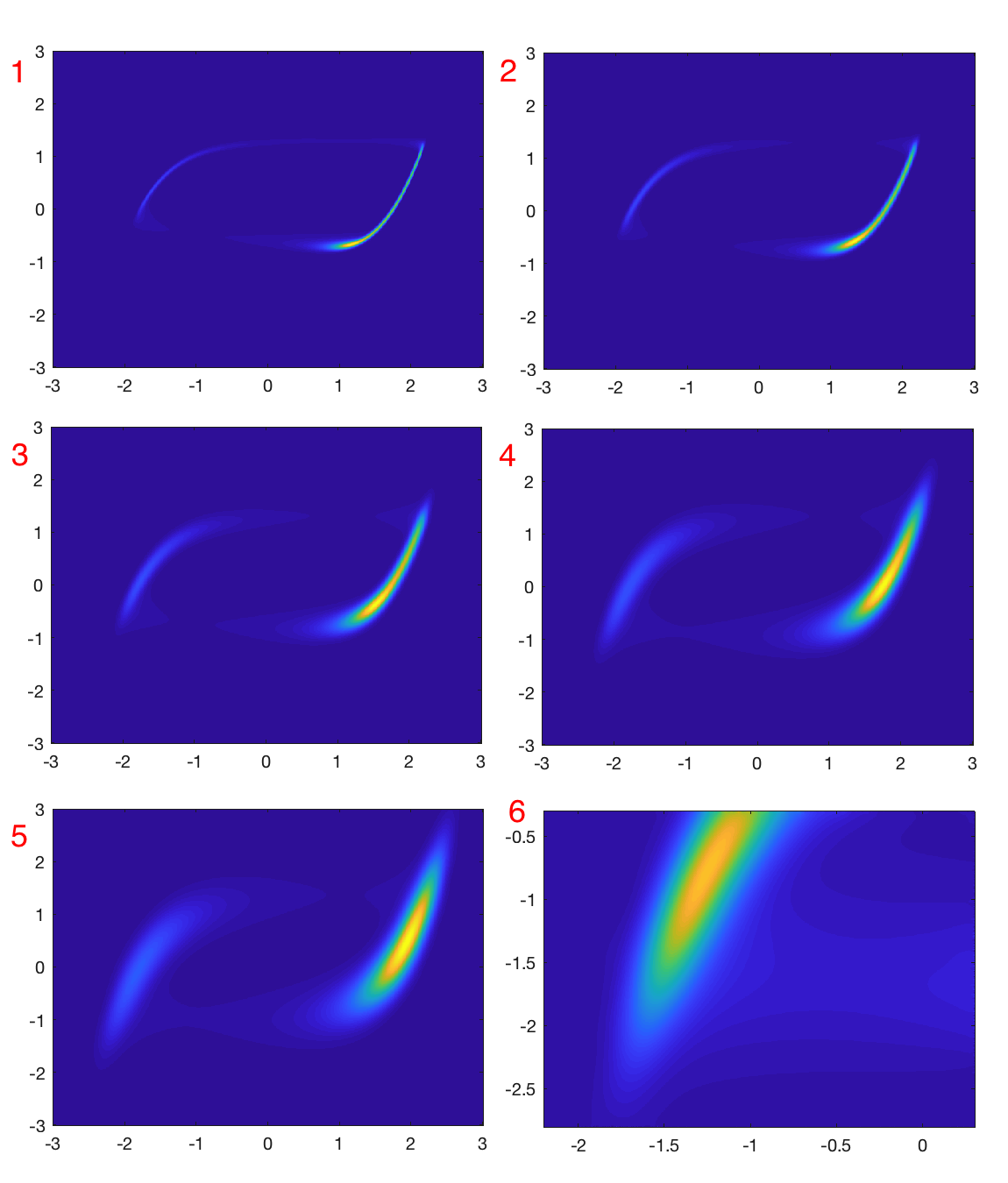}}
\caption{Numerical solutions of the invariant probability density
  function with different $\sigma$. Panel 1 to 6: (1), $\sigma = 0.1$;
  (2), $\sigma = 0.2$; (3), $\sigma = 0.4$; (4), $\sigma = 0.7$; (5),
  $\sigma = 1.0$; (6) local solution of $\sigma = 1.0$ showing the
  lower left ``tail''.}
\label{canardlocal}
\end{figure}

\subsection{3D chaotic oscillators under random perturbations}
The hybrid method demonstrates its full strength in 3D systems. In
this subsection, we compute numerical invariant probability 
density functions of two randomly perturbed chaotic systems: the Lorenz oscillator and
the R\"ossler attractor. Both of them are typical chaotic oscillators
that play a significant role in the study of nonlinear physics and
dynamical systems \cite{rossler1976equation, gaspard2005rossler,
  afraimovich1977origin, lorenz1963deterministic}. 

For the Lorenz system, we mean
\begin{eqnarray}
\label{Lorenz}
 \dot{x}& =  & a (y - x)\\\nonumber
\dot{y} &=&x(b - z) - y \\\nonumber
\dot{z} &=& xy - c z 
\end{eqnarray}
with typical parameters $a = 10$, $b = 28$, and $c = 8/3$. It is well known that system \eqref{Lorenz} has a butterfly-shape
strange attractor. The R\"ossler attractor is a chaotic oscillator
that has the similar mechanism as the Lorenz oscillator. We have
\begin{eqnarray}
\label{Rossler}
 \dot{x}& =  & -y - z\\\nonumber
\dot{y} &=&x + ay \\\nonumber
\dot{z} &=& b + z(x-c) \,.
\end{eqnarray}
Again, we use typical parameters $a = 0.2, b = 0.2$, and $c = 5.7$. 

We are interested in invariant probability density functions of random
perturbations of system \eqref{Lorenz} and system \eqref{Rossler}. In
both systems, a perturbation term $\sigma \mathrm{d} \mathbf{W}_{t}$
is added to the deterministic part, where $\sigma > 0$ is the strength
of noise, and $\mathbf{W}_{t}$ is the standard Wiener process in
$\mathbb{R}^{3}$. Needless to say, it is extremely difficult to solve
a steady state Fokker-Planck equation in 3D on a large domain. Take
the Lorenz oscillator as an example. If the numerical domain has to cover the
attractor, we will solve a 3D equation on a cube $[-25, 25]\times
[-25, 25]\times [0, 50]$. When the grid size is $0.05$, the resultant
numerical solution will have $10^{9}$ grid points. Solving
such a large linear system is very computationally expensive. 

The hybrid method provides an approach to solve a 3D steady state Fokker-Planck
equation locally with high resolution and low cost. If the global
solution is still necessary, one can numerically solve many local
solutions and ``glue'' them together. In Figure \ref{det}, we choose a small box on the attractor as the domain (the
red rectangle). The attractor is projected to the XY-plane for the purpose of
demonstration. Note that the Lorenz oscillator is rotated by a
rotation matrix for the purpose of easier demonstration and mesh
generation. The heights of both domains are $1.0$. Then we use our
hybrid method to compute the invariant probability density
function. The strength of noise is chosen to be $\sigma = 0.3$ for the
Lorenz attractor and $\sigma = 0.1$ for the R\"ossler attractor. We
use larger $\sigma$ because the Lorenz system has a much bigger
attractor. The grid size is $0.05$ for both examples. 

The numerical solution is demonstrated in Figure \ref{rand}, in
which the solution is integrated with respect to $z$ for the purpose
of easier visualizations. We can see that both
invariant probability density functions reveal lots of fine structures of the
strange attractors. And the probability density is higher near the center
of the attractor. In constrast to the very computationally expensive global problem,
it only takes a laptop about $10$ minutes to generate such a local
solution on MATLAB.

\begin{figure}
  \begin{minipage}[b]{0.5\linewidth}
\includegraphics[width=\linewidth]{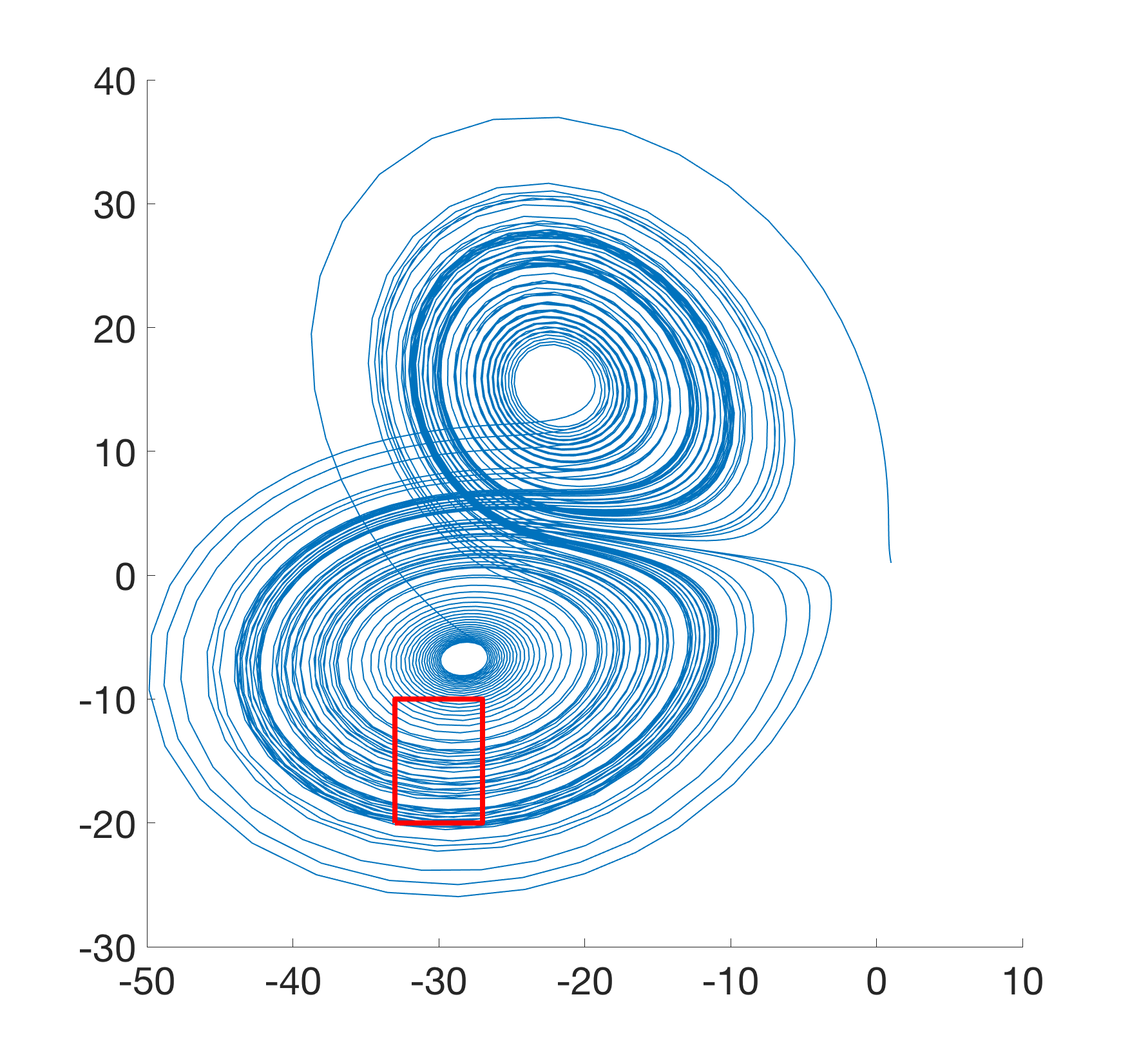}
  \end{minipage}%
  \begin{minipage}[b]{0.5\linewidth}
  \includegraphics[width=\linewidth]{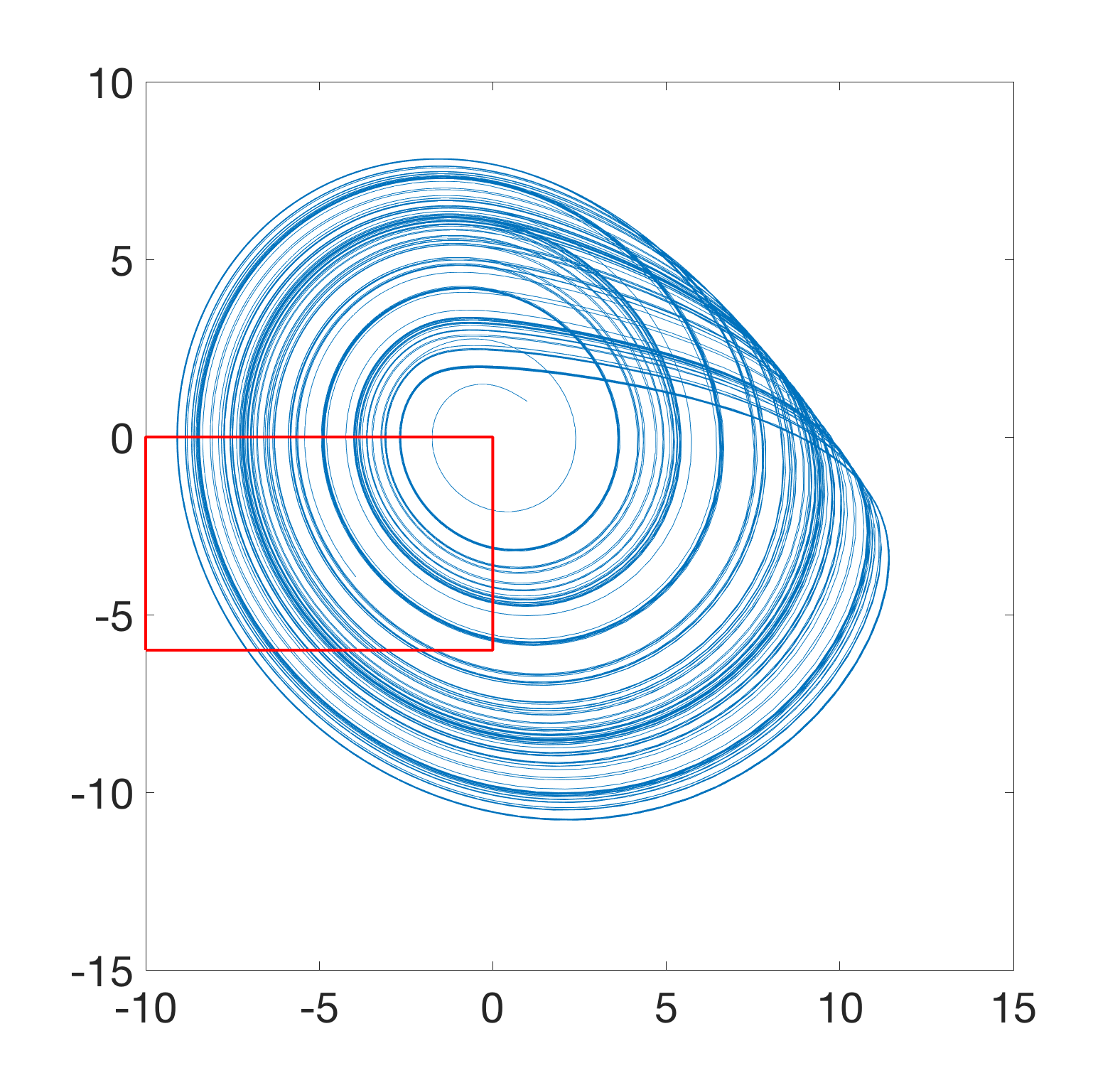}

  \end{minipage}
  \caption{Left: Trajectory of Lorenz oscillator. Right: Trajectory of
    the R\"ossler oscillator. Both: Solution is projected to the
    XY-plane. Red box is the domain of the numerical solution.}
\label{det}
\end{figure}

\begin{figure}
  \begin{minipage}[b]{0.4\linewidth}
\includegraphics[width=\linewidth]{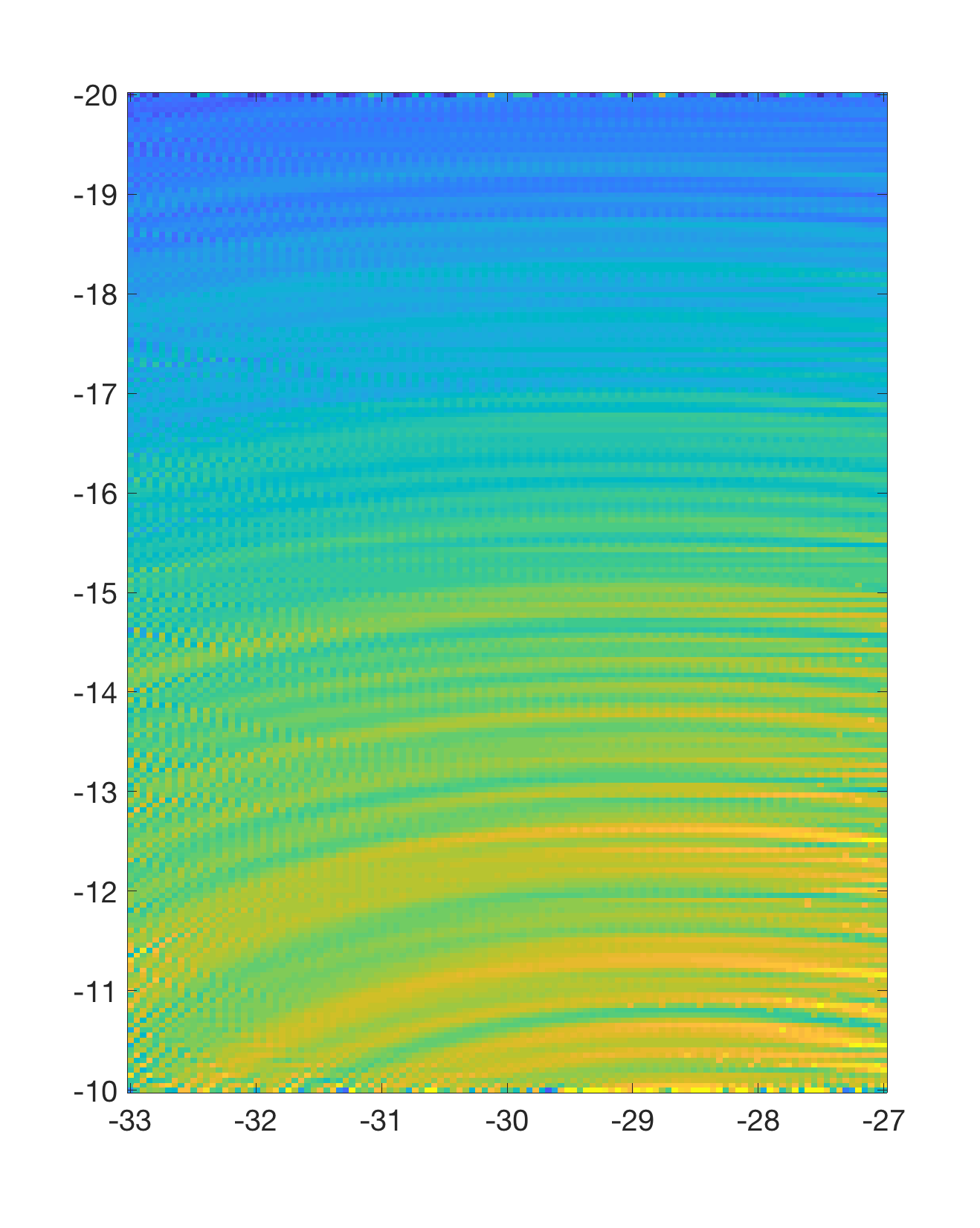}
  \end{minipage}%
  \begin{minipage}[b]{0.6\linewidth}
  \includegraphics[width=\linewidth]{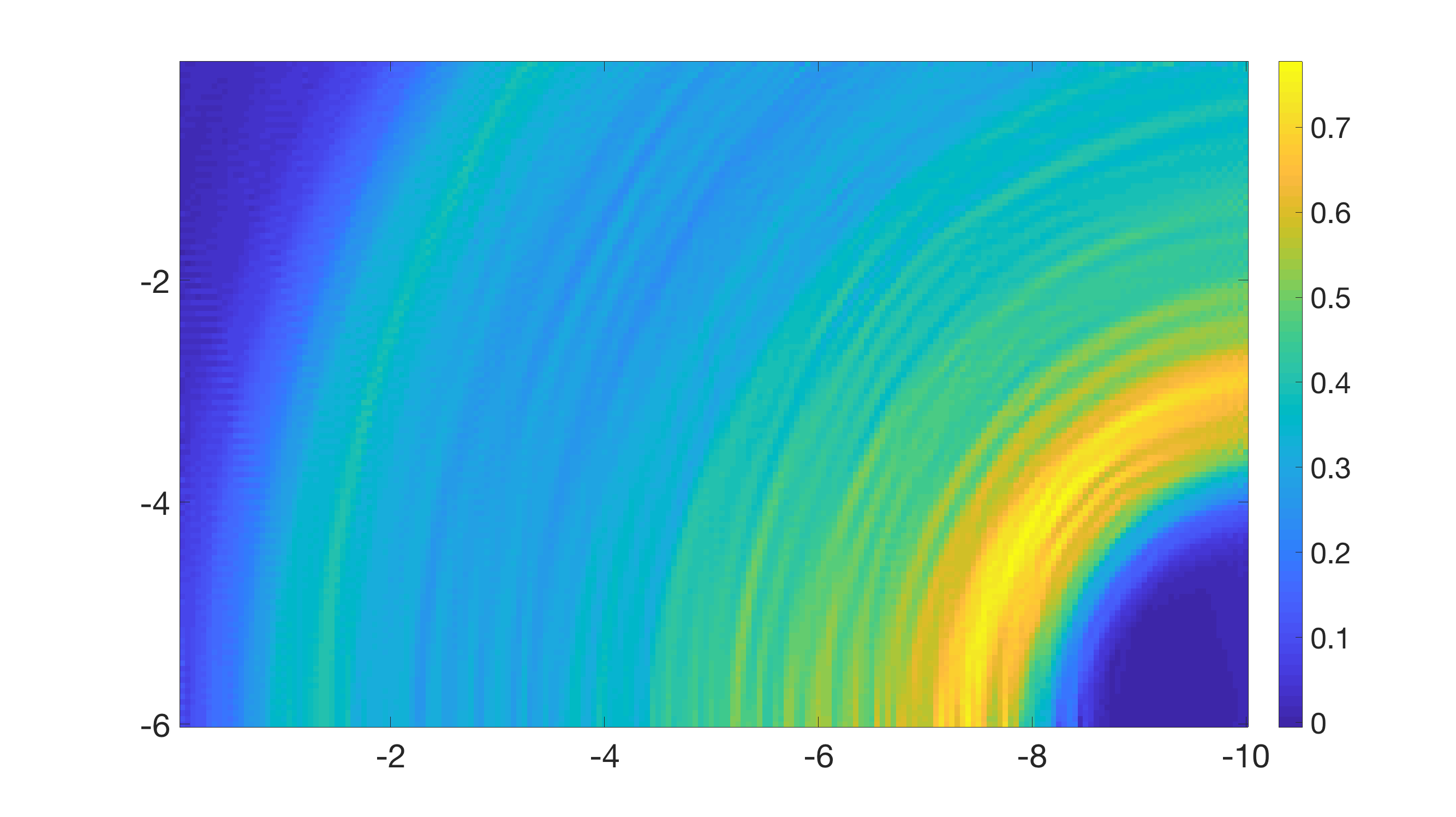}
  \end{minipage}
  \caption{The invariant probability density function of the randomly
    perturbed system. Left panel: Invariant probability density
    function of the Lorenz oscillator with $\sigma =
    0.3$. Right panel: Invariant probability density function of the R\"ossler oscillator with $\sigma = 0.1$.  The numerical solution is projected to the XY-plane for the
    purpose of easier visualization. }
\label{rand}
\end{figure}

\section{Conclusion}

In this paper we present a hybrid numerical method that solves the
steady state Fokker-Planck equation. The numerical discretization
scheme (finite difference scheme, finite element scheme, or Galerkin
method) without boundary condition gives a linear constraint. A
low-accuracy numerical solution produced by the Monte Carlo simulation (or
other variants) serves as a reference solution. The problem is then
converted to an optimization problem, which looks for the least squares
solution with respect to the reference solution, under the linear
constraint given by the numerical discretization scheme. 

The main advantage of this hybrid approach is that it drops the
dependence on boundary conditions. Hence we can compute the steady
state Fokker-Planck equation in any local area, while the traditional
numerical PDE approach has to use a large enough domain in order to apply the
zero boundary condition. This makes a significant difference if one
wants to study the invariant probability density function in a local
area in the vicinity of a strange attractor. The Monte Carlo simulation gives both
good flexibility and some limitations to this hybrid method. Our
simulation shows that the
hybrid method can
tolerate ``local'' fluctuations in the reference solution very well. However, if the
Monte Carlo simulation result has significant and systemic bias, the hybrid method cannot
completely recover the invariant probability density function. Also, the Monte Carlo
simulation usually only generates very few samples (or no sample) in regions that are very far
away from the attractor. In these regions, the hybrid method cannot
recover the tail effectively, while the traditional PDE solver usually
performs better.

This paper serves as the first paper of a series of
investigations. Under this data-driven framework, lots of improvements
can be made to both the Monte Carlo simulation and the numerical PDE
approach. For example, because our data-driven framework does not rely on
boundary conditions, one can use divide-and-conquer strategy to divide the domain into many ``blocks''. Our preliminary work shows
that this approach can significantly accelerate the computation. Another potential improvement is to use some recently
developed sampling techniques in the Monte Carlo simulation. In particular, if
the noise term is large enough, the Gaussian mixture method reported
in \cite{chen2017beating,
  chen2018efficient} can significantly reduce the cost of Monte Carlo
simulations for high-dimensional problems. We expect to incorporate
this sampling technique into our framework in the future.

It remains to comment on the extension to the time dependent
Fokker-Planck equation, as the time evolution of the probability
density function is very importent in many applications. Our
hybrid method can be extended to the time-dependent case after
minor modifications. To use the hybrid method, at each time step, the classical PDE solver
should be replaced by a least squares optimization problem. More
precisely, let $\mathbf{u}_{n}$ denote the numerical solution to a
time dependent Fokker-Planck equation at time step $n$. Then the
numerical solution $\mathbf{u}_{n+1}$ at the next step is obtained by solving
an optimization problem
\begin{eqnarray}
\label{lsqt}
 & \mbox{min} & \| \mathbf{u}_{n+1} - \mathbf{v}_{n+1} \| \\\nonumber
&\mbox{subject to} & \mathbf{A} \mathbf{u}_{n+1} = \mathbf{b}_{n} \,,
\end{eqnarray}
where the linear contraint $\mathbf{A} \mathbf{u}_{n+1} = \mathbf{b}_{n}$ comes
from the numerical discretization scheme (such as implicit Euler scheme or Crank-Nicolson
scheme) for the time dependent Fokker-Planck
equation, and $\mathbf{v}_{n+1}$ is a probability density function
(with lower accuracy) generated by the Monte Carlo
simulation.

\section{Acknowledgement}
The author thanks his former students Ms. Lily Chou and Ms. Huangyi
Shi for some preliminary numerical simulations related to this project.

\bibliography{myref}
\bibliographystyle{amsplain}
\end{document}